\documentclass{amsart}

\usepackage{color,amsmath,amssymb,setspace,savesym,hyperref}
\usepackage[all]{xy}
\savesymbol{cir}
\usepackage[shortalphabetic]{amsrefs}

\renewcommand{\baselinestretch}{1.3}
\renewcommand\MR[1]{%
    \relax\ifhmode\unskip\spacefactor3000 \space\fi
    MR\nolinebreak[3]\hspace{.16667em}#1%
}

\theoremstyle{plain}
\newtheorem{theorem}{Theorem}[section]
\newtheorem*{theorem*}{Theorem}
\newtheorem{lemma}[theorem]{Lemma}
\newtheorem{proposition}[theorem]{Proposition}
\newtheorem{corollary}[theorem]{Corollary}
\newtheorem*{corollary*}{Corollary}

\theoremstyle{remark}
\newtheorem{remark}[theorem]{Remark}

\theoremstyle{definition}
\newtheorem{definition}[theorem]{Definition}
\newtheorem{example}[theorem]{Example}

\newcommand{\Z}{\mathbb{Z}}
\newcommand{\R}{\mathbb{R}}
\newcommand{\dis}{\displaystyle}
\newcommand{\al}{\alpha}
\newcommand{\be}{\beta}
\newcommand{\ga}{\gamma}

\newcommand{\ep}{\epsilon}
\newcommand{\la}{\lambda}

\newcommand{\G}{\Gamma}

\newcommand{\I}{^{-1}}
\newcommand{\gen}{\langle}
\newcommand{\by}{\rangle}
\DeclareMathOperator{\SL}{SL}

\DeclareMathOperator{\Sp}{Sp}
\DeclareMathOperator{\PSp}{PSp}

\DeclareMathOperator{\SO}{SO}

\DeclareMathOperator{\Spin}{Spin}

\numberwithin{equation}{section}

\begin{document}

\title{Representatives of elliptic Weyl group elements in algebraic groups}

\author{Matthew C. B. Zaremsky}
\address{Department of Mathematics \\
Bielefeld University \\
Bielefeld, Germany 33615}
\email{zaremsky@math.uni-bielefeld.de}

\begin{abstract}
\singlespacing
An element $w$ of a Weyl group $W$ is called \emph{elliptic} if it has no
eigenvalue 1 in the standard reflection representation. We determine the order
of any representative $g$ in a semisimple algebraic group $G$ of an elliptic
element $w$ in the corresponding Weyl group $W$. In particular if $w$ has order
$d$ and $G$ is simple of type different from $C_n$ or $F_4$, then $g$ has order
$d$ in $G$.
\end{abstract}

\maketitle

\section{Introduction}
\label{sec:intro}

An element $w$ of a Weyl group $W$ is called \emph{elliptic} if it has no
eigenvalue 1 in the standard reflection representation. It is well known that
the Coxeter elements provide examples of such elements, but in general they are
not the only examples \citelist{\cite{fedotov09}*{Proposition~8}
\cite{humphreys92}*{Lemma~3.16}}. If we think of $W$ not as the Weyl group of a
root system but as the quotient $W=N_G(T)/T$ in some semisimple algebraic group
$G$ with maximal torus $T$, the natural question arises whether representatives
in $G$ of elliptic elements have any nice properties. In this paper we determine
the order of any representative in $G$ of an elliptic element.

The classification of conjugacy classes in Weyl groups is provided in
\cite{carter72}, where they are essentially classified by certain
\emph{admissible diagrams}, which we will call \emph{Carter diagrams}. These
diagrams are particularly useful in the present context since they make it easy
to single out elliptic elements. The question of determining the order of
representatives of elliptic elements was analyzed in \cite{fedotov09}, with some
substantial results in certain cases. The cases of $E_6$ and $E_7$, however,
proved particularly troublesome in that paper, and in the classical cases the
focus was on the case when $G$ is simple. In the present work, instead of
analyzing the problem thinking of $G$ as a matrix group, we use Chevalley
generators and relations to calculate the order of any representative of an
elliptic element. One surprising result is that if $G$ is simple and $w$ is
elliptic with order $d$, then representatives $g$ of $w$ almost always have
order $d$, with the only counterexamples arising in $C_n$ and $F_4$. A summary
of results is given in Table~\ref{tab:final_chart}; see Definition~\ref{spin}
for an explanation of the terminology in the table.

Weyl group elements with no eigenvalue 1 are called \emph{elliptic} in
\cite{lusztig10} and \emph{generalized Coxeter elements} in \cite{dwyer99}. Here
we will generally stick with ``elliptic." If $w$ is elliptic we will also refer
to the conjugacy class of $w$ in $W$ as ``elliptic" since eigenvalues are
conjugation invariant. Our main references for facts about root systems and
semisimple groups are \cite{carter05} and \cite{gls98}. We will use the
numbering of the nodes of the Dynkin diagrams given in \cite{carter05}. (Note
that the numbering for $E_7$ and $E_8$ is different than that given in
\cite{gls98}.)

$$A_{n-1}~~~(n>1)\hspace{.3in}\xy
(2,2.5)*{1}; (11,2.5)*{2}; (29,2.5)*{n-2}; (39,2.5)*{n-1};
(2,0)*{\circ}; (11,0)*{\circ} **\dir{-}; (20,0)*{~\cdots~} **\dir{-};
(29,0)*{\circ} **\dir{-}; (39,0)*{\circ} **\dir{-};
\endxy$$

$$B_n~~~(n>1)\hspace{.3in}\xy
(2,2.5)*{1}; (11,2.5)*{2}; (29,2.5)*{n-1}; (38,2.5)*{n};
(2,0)*{\circ}; (11,0)*{\circ} **\dir{-}; (20,0)*{~\cdots~} **\dir{-};
(29,0)*{\circ} **\dir{-}; {\ar@2{->}(29,0)*{\circ};(38,0)*{\circ}};
\endxy$$

$$C_n~~~(n>1)\hspace{.3in}\xy
(2,2.5)*{1}; (11,2.5)*{2}; (29,2.5)*{n-1}; (38,2.5)*{n};
(2,0)*{\circ}; (11,0)*{\circ} **\dir{-}; (20,0)*{~\cdots~} **\dir{-};
(29,0)*{\circ} **\dir{-}; {\ar@2{<-}(29,0)*{\circ};(38,0)*{\circ}};
\endxy$$

$$D_n~~~(n>3)\hspace{.3in}\xy
(2,2.5)*{1}; (11,2.5)*{2}; (28,2.7)*{n-2}; (38,-5.5)*{n}; (38,5.5)*{n-1};
(2,0)*{\circ}; (11,0)*{\circ} **\dir{-}; (20,0)*{~\cdots~} **\dir{-};
(29,0)*{\circ} **\dir{-}; (38,-3)*{\circ} **\dir{-}; (29,0)*{\circ};
(38,3)*{\circ} **\dir{-};
\endxy$$

$$E_6\hspace{.3in}\xy
(2,2.5)*{1}; (11,2.5)*{2}; (20,2.5)*{3}; (29,2.5)*{5}; (38,2.5)*{6};
(22,-6)*{4};
(2,0)*{\circ}; (11,0)*{\circ} **\dir{-}; (20,0)*{\circ} **\dir{-};
(29,0)*{\circ} **\dir{-}; (38,0)*{\circ} **\dir{-}; (20,0)*{\circ};
(20,-6)*{\circ} **\dir{-};
\endxy$$

$$E_7\hspace{.3in}\xy
(2,2.5)*{1}; (11,2.5)*{2}; (20,2.5)*{3}; (29,2.5)*{4}; (38,2.5)*{6};
(47,2.5)*{7}; (31,-6)*{5};
(2,0)*{\circ}; (11,0)*{\circ} **\dir{-}; (20,0)*{\circ} **\dir{-};
(29,0)*{\circ} **\dir{-}; (38,0)*{\circ} **\dir{-}; (47,0)*{\circ} **\dir{-};
(29,0)*{\circ}; (29,-6)*{\circ} **\dir{-};
\endxy$$

$$E_8\hspace{.3in}\xy
(2,2.5)*{1}; (11,2.5)*{2}; (20,2.5)*{3}; (29,2.5)*{4}; (38,2.5)*{5};
(47,2.5)*{7}; (56,2.5)*{8};  (40,-6)*{6};
(2,0)*{\circ}; (11,0)*{\circ} **\dir{-}; (20,0)*{\circ} **\dir{-};
(29,0)*{\circ} **\dir{-}; (38,0)*{\circ} **\dir{-}; (47,0)*{\circ} **\dir{-};
(56,0)*{\circ} **\dir{-}; (38,0)*{\circ}; (38,-6)*{\circ} **\dir{-};
\endxy$$

$$F_4\hspace{.3in}\xy
(2,2.5)*{1}; (11,2.5)*{2}; (20,2.5)*{3}; (29,2.5)*{4};
(2,0)*{\circ}; (11,0)*{\circ} **\dir{-};
{\ar@2{->}(11,0)*{\circ};(20,0)*{\circ}}; (20,0)*{\circ} ; (29,0)*{\circ}
**\dir{-};
\endxy$$

$$G_2\hspace{.3in}\xy
(2,2.5)*{1}; (11,2.5)*{2};
{\ar@3{->}(2,0)*{\circ};(11,0)*{\circ}}
\endxy$$

\section{Preliminary results}\label{sec:general}
Let $\Phi$ be a reduced crystallographic root system with Weyl group $W$, $K$ an
algebraically closed field, and $G$ a semisimple algebraic $K$-group with root
system $\Phi$. Let $G_u$ be the corresponding universal group and $G_a$ the
adjoint group, as in \cite{gls98}*{Theorem~1.10.4}. Then we have epimorphisms
$G_u\rightarrow G\rightarrow G_a$ with $\dis\ker\left(G_u\rightarrow
G_a\right)=Z(G_u)$ finite. In fact, $G$ is always either $G_a$ or $G_u$
unless $\Phi$ has type $A_n$ or $D_n$. Thus it is almost sufficient to just analyze $G_a$ and $G_u$.

We will need to think of $G$ in terms of Chevalley generators and relations, so
we now establish some facts in that vein. Let $T$ be a maximal torus in $G$ and
let $x_{\al}(\la)$ denote the standard Chevalley generators, where $\al\in\Phi$
and $\la\in K$. For each $\al\in\Phi$, $\la\in K^*$ define
$m_{\al}(\la):=x_{\al}(\la)x_{-\al}(-\la\I)x_{\al}(\la)$ and
$h_{\al}(\la):=m_{\al}(\la)m_{\al}(-1)$. Let $N:=\gen m_{\al}(\la)\by$, and note
that $T=\gen h_{\al}(\la)\by$ \cite{gls98}*{Theorem~1.12.1}. It is a fact that
$N/T\cong W$; see \cite{steinberg67}*{Lemma~22}. The following Chevalley
relation, which we will need later, is established in the proof of
\cite{carter72book}*{Lemma~7.2.2}.

\textbf{(CR1):} For $\al,\be\in\Phi$,
$m_{\al}(1)m_{\be}(1)m_{\al}(1)\I=m_{s_{\al}\be}(c(\al,\be))$ where
$c(\al,\be)=\pm1$ is determined only by $\al$ and $\be$.

This sign $c(\al,\be)$ can sometimes be computed just from knowing the
$\al$-chain of roots through $\be$. As we will see in
Lemma~\ref{orth_m_commute}, if $\al$ and $\be$ are orthogonal then $c(\al,\be)$
is ``usually" 1, and by orthogonality $s_{\al}\be=\be$ so then $m_{\al}(1)$ and
$m_{\be}(1)$ actually commute. Details are given in Lemma~\ref{orth_m_commute}.
An immediate corollary to (CR1) is the following, which does not depend on
$c(\al,\be)$:

\textbf{(CR2):} For $\al,\be\in\Phi$,
$m_{\al}(1)h_{\be}(-1)m_{\al}(1)\I=h_{s_{\al}\be}(-1)$.

We now define $N_0$ to be $\gen m_{\al}(1)\mid\al\in\Phi\by$ and $T_0$ to be
$\gen h_{\al}(-1)\mid\al\in\Phi\by$. It is easy to see that $N_0/T_0\cong W$, by
the same proof that $N/T\cong W$. See also \cite{abr11}*{Lemma~4.2} and
\cite{gls98}*{Remark~1.12.11}. Since $T_0$ is abelian and all its elements
square to 1, we immediately see that any Weyl group element $w$ of order $d$ has
at least one representative $g_0$ of order either $d$ or $2d$.

For elliptic $w$, by \cite{fedotov09}*{Theorem~1} and, independently,
\cite{abr11}*{Theorem~4.3}, all representatives of $w$ in $N$ have the same
order. In fact by the proof of \cite{abr11}*{Theorem~4.3}, for any
representative $g$, $g^d=g_0^d$. Thus to determine the order of any
representative $g$ of an elliptic Weyl group element $w$ with order $d$, it
suffices to check whether $g_0^d=1$ or not, for $g_0\in N_0$ representing $w$.
We encode this fact into the following proposition, which is proved in the
sources mentioned above.

\begin{proposition}\label{reps_have_same_order}
Let $w\in W$ be an elliptic element with order $d$. Then all representatives of
$w$ in $N$ have the same order. In particular they all have order $d$ if
$g_0^d=1$ and order $2d$ otherwise.
\end{proposition}

\begin{remark}\label{torsion_converse}
The converse of Proposition~\ref{reps_have_same_order} is also true for most
$K$; that is, if $w$ is not elliptic and if $K$ contains an element of infinite
order, then $w$ has a representative of infinite order in $N$. This is proved in
Theorem~4.3 in \cite{abr11}, but we will not need this fact here.
\end{remark}

The elliptic elements are classified in \cite{carter72}, and to each conjugacy
class of elliptic elements is assigned an ``admissible diagram" $\G$, which we
call a \emph{Carter diagram}. For $w\in W=W(\Phi)$, we can always find linearly
independent roots $\be_1,\dots,\be_r$ such that $w=s_{\be_1}\cdots s_{\be_r}$,
and $w$ is elliptic if and only if $r=n$ where $n$ is the rank of $\Phi$
\cite{carter72}. In general $\G$ is constructed by taking a node for each
$\be_i$ and connecting $\be_i$ to $\be_j$ with a certain number of edges given
by the same rule as for Dynkin diagrams (that is, depending on the angle between
$\be_i$ and $\be_j$). In particular if $\al_1,\dots,\al_n$ are the simple roots
then $w=s_{\al_1}\cdots s_{\al_n}$ is a Coxeter element and simply has Carter
diagram equal to the Dynkin diagram of $\Phi$. Another important case is when
$\Phi$ contains $n$ mutually orthogonal roots $\be_1,\dots,\be_n$. In this case
$w=s_{\be_1}\cdots s_{\be_n}$ is the negative identity element $-I$ in $W$, with
Carter diagram $A_1^n$, i.e., $n$ unconnected nodes. It is possible that two
elements in $W$ can have the same Carter diagram without being conjugate, but this will never happen for \emph{elliptic} elements \cite{carter72}.

At this point for the sake of brevity we introduce the following definitions:

\begin{definition}\label{spin}
Let $w\in W$ be elliptic with order $d$. If all representatives of $w$ in $G$
have order $d$ we say $w$ has \emph{spin 1}. If all representatives of $w$ have
order $2d$ we say $w$ has \emph{spin $-1$}. Note that this is a property of $w$
and of $G$, not just of $w$. Thus we will often need to refer to
\emph{$G$-spin}, \emph{adjoint spin}, or \emph{universal spin}. Spin is of
course preserved by conjugation, so we may also refer to the spin of a conjugacy
class or Carter diagram. Furthermore, if $w\in W$ is elliptic with order $d$ and
$g_0\in N_0$ represents $w$, we will call $g_0^d\in T_0$ the \emph{spin
signature} of $w$. This doesn't depend on the choice of $g_0$, and so is well
defined.
\end{definition}
Unlike spin, the spin signature may not be conjugation invariant. In practice we
will often find that the spin signature of $w$ is central in $G$, in which case
we can refer to the spin signature of a conjugacy class or Carter diagram. In
Section~\ref{sec:cox_elts} we will present a labeling of the Carter diagram of
$w$ that helps to calculate the spin signature. First we establish a few results
that simplify things considerably.

\begin{corollary}\label{odd_order}
Let $w\in W$ be elliptic with odd order $d$. Then $w$ has spin 1.
\end{corollary}
\begin{proof}
Let $g_0\in N_0$ represent $w$, so $g_0^d\in T_0$. Since $(g_0g_0^d)^d=g_0^d$,
in fact $g_0^{d^2}=1$. But since $d$ is odd this means that $g_0$ cannot have
order $2d$, and so has order $d$.
\end{proof}

\begin{lemma}\label{powers_spin}
For any $w\in W$ and $r\in\Z$, if $w$ and $w^r$ are both elliptic then $w$ has
the same spin and spin signature as $w^r$.
\end{lemma}
\begin{proof}
Say $w$ has order $d$ and spin 1. Then any representative $g$ of $w$ in $N$ has
order $d$, so $g^r$ has the same order as $w^r$ implying that $w^r$ has spin 1.
The spin $-1$ case follows by a parallel argument, and the fact that the spin
signatures are the same is immediate.
\end{proof}

\begin{definition}
Let $w_1$ and $w_2$ be elements of $W$. If there exists $r\in\Z$ such that
$w_1^r=w_2$ or $w_2^r=w_1$ then we will call $w_1$ and $w_2$ \emph{linked}.
Similarly we may refer to the corresponding conjugacy classes as \emph{linked}
if there exist representatives from each class that are linked. The point is
that linked classes have equal spins, and linked elements have equal spins and
spin signatures.
\end{definition}

To tell whether two elliptic classes are linked we will often make use of
Table~3 in \cite{carter72}, which lists the characteristic polynomials of
elliptic elements. Knowing the eigenvalues of an elliptic element $w$ allows us
to easily check which powers $w^r$ are elliptic, and to identify the conjugacy
class of $w^r$. For example if the eigenvalues of $w$ are all primitive
$2r_{th}$ roots of unity, then $w^r=-I$ and $w$ is linked to $-I$.

\begin{lemma}\label{lands_in_center}
Suppose $-I\in W$. Then any representative $g$ of $-I$ in $N$ satisfies $g^2\in
Z(G)$. In particular if $G$ is simple then $-I$ has spin 1.
\end{lemma}
\begin{proof}
Since $g^2=g_0^2$ for $g_0\in N_0$ representing $-I$, without loss of generality
$g\in N_0$. By \cite{carter72book}*{Lemma~7.2.1(i)}, we thus have
$gx_{\al}(\la)g\I=x_{-\al}(\ep_{\al}\la)$, where $\ep_{\al}=\pm1$ depends on $g$
and $\al$ but not on $\la$ or $G$. Similarly
$g^2x_{\al}(\la)g^{-2}=x_{\al}(\ep_{\al}\ep_{-\al}\la)$. By
\citelist{\cite{carter72book}*{Proposition~6.4.3}
\cite{steinberg67}*{Lemma~19(a)}} however, $\ep_{-\al}=\ep_{\al}$, and so
actually $g^2x_{\al}(\la)g^{-2}=x_{\al}(\la)$. Since the $x_{\al}(\la)$ generate
$G$, as explained in \cite{gls98}*{Remark~1.12.3}, indeed $g^2\in Z(G)$.
\end{proof}

\begin{corollary}\label{simple_-I}
If $G$ is simple then any elliptic element $w$ of $W$ that is linked to $-I\in
W$ has spin 1.\qed
\end{corollary}

We can calculate the spin of many elliptic conjugacy classes using just
Corollaries~\ref{odd_order} and \ref{simple_-I}. For those classes that cannot
be dealt with using just these two corollaries, we need to do a bit of
computation. To help with this we transcribe a version of Table~1.12.6 in
\cite{gls98}, listing all elements of order 2 in $Z(G_u)$. We will use the
numbering of the simple roots $\al_i$ given in Section~\ref{sec:intro}. For each
$i=1,\dots,n$ let $h_i:=h_{\al_i}(-1)$.

\begin{table}[h!]
\caption{Central elements of order 2}
\begin{tabular}{|c|c|}
\hline
$\Phi$ & elements of order 2 in $Z(G_u)$ \\
\hline
$A_{n-1}$ ($n$ even) & $h_1h_3\cdots h_{n-1}$ \\
\hline
$A_{n-1}$ ($n$ odd) & none \\
\hline
$B_n$ & $h_n$ \\
\hline
$C_n$ & $h_1h_3\cdots h_k$; $k=2\left\lfloor\frac{n-1}{2}\right\rfloor+1$ \\
\hline
$D_{2\ell}$ & $h_1h_3\cdots h_{2\ell-1}$, $h_{2\ell-1}h_{2\ell}$, $h_1h_3\cdots
h_{2\ell-3}h_{2\ell}$ \\
\hline
$D_{2\ell+1}$ & $h_{2\ell}h_{2\ell+1}$ \\
\hline
$E_6$ & none \\
\hline
$E_7$ & $h_1h_3h_5$ \\
\hline
~ & In all other cases $Z(G_u)=1$ \\
\hline
\end{tabular}
\label{tab:center}
\end{table}

\section{Spin signatures of Coxeter elements}\label{sec:cox_elts}
At this point we declare that we only consider fields $K$ with characteristic
different than 2. If the characteristic is 2 then $T_0=\{1\}$, and all elliptic
elements have spin 1, so this case is trivial.

\begin{lemma}\label{orth_m_commute}
Let $\al,\be$ be orthogonal roots in a root system $\Phi$. If the $\al$-chain of
roots through $\be$ is just $\be$, then $[m_{\al},m_{\be}]=1$.
\end{lemma}
\begin{proof}
Since the $\al$-chain of roots through $\be$ is just $\be$, $c(\al,\be)=1$ by
the proof of \cite{carter72book}*{Proposition~6.4.3}. Since $\al,\be$ are
orthogonal, $s_{\al}\be=\be$, and so by (CR1) indeed $[m_{\al},m_{\be}]=1$.
\end{proof}
This holds for example if $\al$ and $\be$ are orthogonal \emph{long} roots.
Also, if $\Phi$ is simply laced then any two orthogonal roots will have this
property. The next lemma is a version of \cite{gls98}*{Theorem~1.12.1(e)}, which
is standard, and we will not prove it here.

\begin{lemma}\label{h_into_simples}
For any root $\al$, if
$\dis\frac{\al}{\gen\al,\al\by}=\sum_{i=1}^nc_i\frac{\al_i}{\gen\al_i,\al_i\by}$
then $h_{\al}(-1)=h_1^{c_1}\cdots h_n^{c_n}$.
\end{lemma}

Let $w\in W$ be elliptic with order $d$, and let $\G$ be its Carter diagram. Let
$\G=\G_1\times\cdots\times\G_r$ be a decomposition of $\G$ into connected
components. Note that roots labeling nodes of different components $\G_i$ are
orthogonal. We know that $w=w_1\cdots w_r$ where each $w_i$ has Carter diagram
$\G_i$ and all the $w_i$ commute with each other (though note that the $w_i$ are
not elliptic). Let $d_i$ denote the order of $w_i$, so $d$ is the least common
multiple of the $d_i$.

\begin{definition}\label{content}
If $\G'$ is the Carter diagram of $w'\in W$ and $w'$ has order $d'$, define the
\emph{content} of $\G'$ to be the power of 2 in the prime factorization of $d'$.
If $\G=\G_1\times\cdots\times\G_r$ and $w=w_1\cdots w_r$ as above then any
$\G_i$ with the same content as $\G$ will be called a \emph{relevant} component.
All other $\G_i$ will be called \emph{irrelevant}.
\end{definition}

The point of this definition is that if $\G_i$ is an irrelevant component, with
$g_i\in N_0$ representing $w_i$, then
$$g_i^d=g_i^{d_i\frac{d}{d_i}}=1$$
since $g_i^{d_i}\in T_0$ and 2 divides $d/d_i$. Thus the irrelevant components
in some sense do not contribute to the spin signature of $w$.

\begin{example}
Consider an elliptic element $w$ in $W=W(C_6)$ with Carter diagram $\G=C_2\times
C_4$. (This exists by \cite{carter72}*{Proposition~24}.) If $w=w_1w_2$ is the
corresponding decomposition, we see that $w_1$ has order 4 and $w_2$ has order
8, so $w$ has order 8 and the $C_2$ factor of $\G$ is irrelevant. If
$g_0=g_1g_2$ is the corresponding representative of $w=w_1w_2$ in $N_0$ then
$g_1^8=1$. It turns out the $g_i$ commute (see Section~\ref{sec:C}), so
$g_0^8=g_2^8$.
\end{example}

To calculate $g_0^d$ in general we need to find a way to calculate these powers
of representatives of relevant components, and then combine them in the correct
way. Let $\G_i$ be a relevant component that is itself a Dynkin diagram, and let
$w_i$ be as above. Without loss of generality we may assume $w_i$ has order $d$.
Let $\be_1,\dots,\be_{n_i}$ be the roots labeling each node. For our
representative of $w_i$ in $N_0$ we take $g_i=m_1\cdots m_{n_i}$, where
$m_j:=m_{\be_j}(1)$. Let $S':=\{s_1,\dots,s_{n_i}\}$ where $s_j=s_{\be_j}$, and
let $W':=\gen S'\by$. Then $W'$ is a Weyl subgroup of $W$ and $(W',S')$ is a
spherical Coxeter system with Coxeter diagram $\G_i$. To calculate $g_i^d$ we
will use (CR2), and a series of relations that are closely related to the
defining relations of $W'$.

First, we make the assumption that if $\be_j$ and $\be_k$ are orthogonal, then
$m_j$ and $m_k$ commute, for all $1\leq j,k\leq n_i$. This assumption will need
to be checked on a case-by-case basis using Lemma~\ref{orth_m_commute}, but at
least if $\Phi$ is simply connected we get it for free. This mirrors the
relation $(s_js_k)^2=1$ in $W'$, and we also have an immediate analogue to the
relation $s_j^2=1$, namely $m_j^2=h_{\be_j}(-1)$, which is just true by
construction. This leaves the braid relations involving non-orthogonal roots.

Let $\be_j$ and $\be_k$ be non-orthogonal roots labeling nodes of $\G_i$. Since
$W'$ is a Weyl subgroup of $W$, $(s_js_k)$ must have order either 3, 4, or 6.
The order 6 case corresponds to a triple edge between the two nodes, which only
appears if $\Phi$ is $G_2$, and this case is easily covered in
Section~\ref{sec:G2} using only Corollary~\ref{odd_order}. As such we can ignore
this case, and assume the nodes have either a single edge or a double edge.

\begin{lemma}\label{braid}
With notation as above, if $\be_j$ and $\be_k$ label nodes connected by a single
edge then $(m_jm_k)^3=1$, and if they label nodes connected by a double edge,
with $\be_k$ the short root, then
$(m_jm_k)^4=(m_km_j)^4=h_{\be_k}(-1)$.
\end{lemma}
\begin{proof}
First suppose it is a single edge, so $(s_js_k)$ has order 3. Note that
$m_km_j=m_{s_k\be_j}(c(\be_k,\be_j))m_k$ by (CR1). Also, by Proposition~6.4.3 in
\cite{carter72book}, $c(\be_j,\be_k)=-c(\be_k,\be_j)$ and
$c(\be_k,\be_j)c(\be_k,s_k\be_j)=-1$. Thus
\begin{align*}
(m_jm_k)^3&= m_j m_{s_j\be_k}(c(\be_k,\be_j))
m_j(c(\be_k,\be_j)c(\be_k,s_k\be_j)) m_k^3\\
&=m_k(c(\be_k,\be_j)c(\be_j,s_k\be_j)) m_j m_j(c(\be_k,\be_j)c(\be_k,s_k\be_j))
m_k^3\\
&=m_km_jm_j\I m_k^3=1.
\end{align*}
In other words, the braid relation $s_js_ks_j=s_ks_js_k$ lifts to
$m_jm_km_j=m_k\I m_j\I m_k\I$ in $N_0$.

Now suppose it is a double edge, so $(s_js_k)$ has order 4, and assume $\be_k$
is the short root. Proposition~6.4.3 in \cite{carter72book} tells us that now
$c(\be_k,\be_j)c(\be_k,s_k\be_j)=1$, and that
$m_k(\la)m_{s_j\be_k}(\mu)=m_{s_j\be_k}(-\mu)m_k(\la)$. By repeated application
of (CR1) we thus get that
\begin{align*}
(m_jm_k)^4&=m_{s_j\be_k}(c(\be_j,\be_k))m_k(c(\be_j,\be_k)c(\be_j,s_j\be_k))m_{
s_j\be_k}(c(\be_j,s_j\be_k))m_km_j^4\\
&=m_{s_j\be_k}(c(\be_j,\be_k))m_{s_j\be_k}(-c(\be_j,s_j\be_k))m_k(c(\be_j,
\be_k)c(\be_j,s_j\be_k))m_k\\
&=h_{\be_k}(-1).
\end{align*}
In other words, the braid relation $s_js_ks_js_k=s_ks_js_ks_j$ lifts to
$$m_jm_km_jm_k=h_{\be_k}(-1)m_k\I m_j\I m_k\I m_j\I$$
in $N_0$. Since $(m_km_j)^4$ just equals $m_k(m_jm_k)^4m_k\I$ we also
immediately get that $(m_km_j)^4=h_{\be_k}(-1)$.
\end{proof}

Note that these relations, plus (CR2), really are sufficient to calculate
$g_i^d$. This is because the corresponding relations in $W'$ are sufficient to
prove $w_i^d=1$, and then (CR2) is enough to identify the correct element of
$T$. It is very important that these relations are completely local, i.e., they
only depend on the roots involved and not on the global structure of $\Phi$, and
in particular don't require us to know the sign of any $c(\al,\be)$. The only
assumption we have made is that any $m_j,m_k$ corresponding to orthogonal
$\be_j,\be_k$ should commute. The fact that these relations only depend on the
roots means that, to calculate $g_i^d$ for a relevant component $\G_i$ that is a
Dynkin diagram, we can actually just calculate the spin signature of a Coxeter
element in the Weyl group with $\G_i$ as its Dynkin diagram. This will work as
long as we choose $g_i$ correctly, i.e., as a product of $m_{\al}(1)$ where
$\al$ ranges over $\G_i$. For example, it turns out that Coxeter elements in
$W(A_3)$ have spin signature $h_1h_3$, and so if some relevant component $\G_i$
is of type $A_3$ and is labeled by $\be_1,\be_2,\be_3$ (and if $m_{\be_1}$ and
$m_{\be_3}$ commute), then $g_i^d=h_{\be_1}(-1)h_{\be_3}(-1)$. We can then use
Lemma~\ref{h_into_simples} to express $g_i^d$ as a product of $h_j$.

To illustrate that we can calculate $g_i^d$ without knowing the various
$c(\al,\be)$, we do the example of $A_3$ here (with $d=4$).

\begin{example}\label{A3_example}
Suppose $\be_1,\be_2,\be_3$ are roots labeling nodes in $\G$ such that
$$\xy
(2,2.5)*{\be_1}; (11,2.5)*{\be_2}; (20,2.5)*{\be_3};
(2,0)*{\circ}; (11,0)*{\circ} **\dir{-}; (20,0)*{\circ} **\dir{-};
\endxy$$
is a connected component of $\G$. Let $m_i=m_{\be_i}(1)$ and assume that $m_1$
and $m_3$ commute. We now show that $(m_1m_2m_3)^4=h_{\be_1}(-1)h_{\be_3}(-1)$,
using only the relations $m_i^2=h_{\be_i}(-1)$, $m_1m_2m_1=m_2\I m_1\I m_2\I$,
$m_2m_3m_2=m_3\I m_2\I m_3\I$, $m_1m_3=m_3m_1$, and (CR2). For brevity we will
write $h_i$ for $h_{\be_i}(-1)$, but note that $\be_i$ probably is different
than the simple root $\al_i$ in $\Phi$.
\begin{align*}
(m_1m_2m_3)^4&=(m_1m_2m_1m_3m_2m_3)^2\\
&=(m_2\I m_1\I m_2\I m_2\I m_3\I m_2\I )^2\\
&=m_2\I m_1\I h_2 m_3\I h_2 m_1\I h_2 m_3\I m_2\I \\
&=h_1(h_1h_2h_3)h_3 m_2\I h_1h_3 m_2\I \\
&=h_2(h_1h_2)(h_2h_3)h_2=h_1h_3
\end{align*}
The last two lines made repeated use of (CR2) and Lemma~\ref{h_into_simples}.
\end{example}

We could theoretically develop an algorithm to calculate the spin signatures of
Coxeter elements for any $\Phi$ in this way, but this would not be a realistic
way to calculate the spin of an arbitrary elliptic element. The point is that
since any such calculations depend only on the roots in the diagram and not on
$\Phi$, we don't have to do this, provided we can calculate the spin signatures
of Coxeter elements some other way. We can in fact do this for $A_{n-1}$, $B_n$,
$C_n$, and $E_7$, and this turns out to be sufficient.

\subsection{Coxeter elements in $A_{n-1}$}\label{sec:ACox}
The results for this case are well-known but we present them for completeness.
By Proposition~23 in \cite{carter72}, the Coxeter elements are the only elliptic
elements in $W=W(A_{n-1})$. Thinking of $W$ as $S_n$, these are precisely the
$n$-cycles. If $n$ is odd then these all have odd order and thus spin 1 by
Corollary~\ref{odd_order}. Suppose now that $n$ is even. By
\cite{gls98}*{Theorem~1.10.7(a)} we know that $G$ is a quotient of $\SL_n(K)$ by
a central subgroup $Z'$. Let $w$ be a Coxeter element and $g_0$ a representative
in $N_0$. Since $w$ is an odd permutation, and $g_0$ has determinant 1, we see
that an odd number of entries of $g_0$ are -1. Thus $g_0^n=-I_n$, and so $w$ has
spin 1 if $-I_n\in Z'$ and spin $-1$ if $-I_n\not\in Z'$. In particular all
elliptic elements of $W(A_{n-1})$ have universal spin $-1$ if $n$ is even. Also
note that when $n$ is even, $-I_n$ is the unique element of order 2 in $Z(G_u)$,
so by Table~\ref{tab:center} the spin signature $g_0^n$ must equal $h_1h_3\cdots
h_{n-1}$. In particular in the $A_3$ case we get $h_1h_3$, as referenced
earlier.

\subsection{Coxeter elements in $B_n$}\label{sec:BCox}
Let $w\in W=W(B_n)$ be a Coxeter element, with order $2n$. Since $w$ has
characteristic polynomial $t^n+1$, $w$ is linked to $-I$, and any representative
of $w$ raised to the $2n$ will equal $g_0^2$, where $g_0$ represents $-I$ in
$N_0$. It thus suffices to calculate $g_0^2$. Let $\{\be_1,\dots,\be_n\}$ be the
orthonormal basis of roots given in \cite{carter05}*{Section~8.3}, so
$-I=s_{\be_1}\cdots s_{\be_n}$ and $g_0=m_1\cdots m_n$ where
$m_i:=m_{\be_i}(1)$. Note that for any $i\neq j$, the $\be_j$-chain of roots
through $\be_i$ is $\be_i-\be_j,\be_i,\be_i+\be_j$, and $\be_i,\be_j$ are
orthogonal. Thus by the proof of
Proposition~\cite{carter72book}*{Proposition~6.4.3},
$m_im_{\be_j}(\ep)m_i\I=m_{\be_j}(-\ep)$ for $\ep=\pm1$. Moreover, it is
straightforward to calculate that for any $i$, $h_{\be_i}(-1)=h_n$. We can now
calculate $g_0^2$.
\begin{align*}
g_0^2&=m_1\cdots m_nm_1\cdots m_n\\
&=m_1m_1((-1)^{n-1})m_2m_2((-1)^{n-2})\cdots m_nm_n((-1)^0)\\
&=h_n^k
\end{align*}
where $\dis k=\left\lfloor\frac{n+1}{2}\right\rfloor$. Since $h_n\in Z(G_u)$ by
Table~\ref{tab:center}, this tells us that $-I$, and thus any Coxeter element,
has adjoint spin 1. In the universal case the spin is 1 if and only if $n$ is
congruent to 0 or 3 modulo 4.

\subsection{Coxeter elements in $C_n$}\label{sec:CCox}
As in the $B_n$ case, we need to calculate $g_0^2$, where $g_0$ represents $-I$.
We claim that $g_0^2=h_1h_3\cdots h_k$ where $\dis
k=2\left\lfloor\frac{n-1}{2}\right\rfloor+1$. Let $\be_1,\dots,\be_n$ denote the
orthonormal basis of $\langle\Phi\rangle_{\R}$ given in
\cite{carter05}*{Section~8.4}, so $-I=s_{2\be_1}\cdots s_{2\be_n}$ and
$g_0=m_1\cdots m_n$ where $m_i:=m_{2\be_i}(1)$. Since the $2\be_i$ are all long
and are mutually orthogonal, they have trivial root chains through each other
and so the $m_i$ all commute by Lemma~\ref{orth_m_commute}. Thus
$g_0^2=h_{2\be_1}(-1)\cdots h_{2\be_n}(-1)$. Now, for each $i$,
$2\be_i=2\al_i+2\al_{i+1}+\cdots+2\al_{n-1}+\al_n$. By
Lemma~\ref{h_into_simples} then, $h_{2\be_i}(-1)=h_ih_{i+1}\cdots h_n$. The
result now follows immediately. As a consequence we see that $-I$, and thus all
Coxeter elements in $W(C_n)$, have universal spin $-1$, and adjoint spin 1 by
Table~\ref{tab:center}.

\subsection{Coxeter elements in $E_7$}\label{sec:E7Cox}
In type $E_7$, the eigenvalues of a Coxeter element are the primitive
$18_{th}$ roots of unity and -1, so a Coxeter element to the $9_{th}$ power
equals $-I$. Since linked elements have the same spin, as before we actually
want to calculate the spin of $-I$. Let $e_1,\dots,e_8$ be an orthonormal basis
of $\R^8$, with the simple roots given by $\al_1=e_1-e_2$, $\al_2=e_2-e_3$,
$\al_3=e_3-e_4$, $\al_4=e_4-e_5$, $\al_5=e_5-e_6$, $\al_6=e_5+e_6$,
$\al_7=\frac{-1}{2}(e_1+\cdots+e_8)$. (This is as in
\cite{carter05}*{Section~8.7}, though we use different notation.) If we then let
$\be_1=e_1-e_2$, $\be_2=e_1+e_2$, $\be_3=e_3-e_4$, $\be_4=e_3+e_4$,
$\be_5=e_5-e_6$, $\be_6=e_5+e_6$, $\be_7=-e_7-e_8$, the $\be_i$ are mutually
orthogonal so $-I=s_{\be_1}\cdots s_{\be_7}$.

Let $g_0=m_1\cdots m_7$ represent $-I$ in $N_0$, where $m_i:=m_{\be_i}(1)$.
Since $E_7$ is simply laced, by Lemma~\ref{orth_m_commute}
$g_0^2=h_{\be_1}(-1)\cdots h_{\be_7}(-1)$. Now, $\be_1=\al_1$,
$\be_2=\al_1+2\al_2+2\al_3+2\al_4+\al_5+\al_6$, $\be_3=\al_3$,
$\be_4=\al_3+2\al_4+\al_5+\al_6$, $\be_5=\al_5$, $\be_6=\al_6$, and
$\be_7=\al_1+2\al_2+3\al_3+4\al_4+2\al_5+3\al_6+2\al_7$, and so by
Lemma~\ref{h_into_simples}, $g_0^2=h_1h_3h_5$. By Table~\ref{tab:center} this is
precisely the non-trivial element of the center. So $-I$ has universal spin $-1$
and adjoint spin 1.

\section{Orders and spin signatures of elliptic elements}\label{sec:elliptic}
Let $w\in W$ be elliptic with order $d$, and let
$\G$ be its Carter diagram. Let $\G=\G_1\times\cdots\times\G_r$ be a
decomposition of $\G$ into connected components. Since we only care about the
roots labeling the nodes of the relevant components inasmuch as they yield a
certain product of $h_j$ according to Lemma~\ref{h_into_simples}, we devise the
following convenient way to label Carter diagrams, which we will call a
\emph{spin labeling}. If a node is labeled with the root $\al$, we re-label it
with the tuple $(i_1,\dots,i_k)$ such that $h_{\al}(-1)=h_{i_1}\cdots h_{i_k}$.
If $\al=\al_i$ is simple, we just maintain the original ``$i$" label. We will
only need to worry about $\G_i$ that are Dynkin diagrams, so the spin signature
$g_i^d$ is just the product of $h_j$ where $j$ ranges in the appropriate way
over the spin labeling of $\G_i$. For instance if $\G_i$ is
$$\xy
(2,2.5)*{(1,3)}; (11,2.5)*{2}; (22,2.5)*{(2,3,4)};
(2,0)*{\circ}; (11,0)*{\circ} **\dir{-}; (22,0)*{\circ} **\dir{-};
\endxy$$
and relevant, then $g_i^d=(h_1h_3)(h_2h_3h_4)=h_1h_2h_4$, since $\G_i$ has type
$A_3$.

These spin labelings are really only useful when $\G$ is the Dynkin diagram of a
Weyl subgroup of $W$, which is equivalent to saying that $\G$ is cycle-free
\cite{carter72}*{Lemma~8}. Luckily, it will turn out that for all examples where
$\G$ has cycles, we can find its spin just using Corollaries~\ref{odd_order} and
\ref{simple_-I}. Since Carter diagrams that are Dynkin diagrams all arise by an
iterated process of removing nodes from extended Dynkin diagrams, we are
especially interested in the spin labeling of nodes corresponding to
$-\widetilde{\al}$, where $\widetilde{\al}$ is a highest root. In the $B_n$ case
it will be convenient to instead use the highest short root $\widetilde{\al_s}$.
We collect here the decompositions of negative highest roots into simple roots,
only for the cases we will actually use.

\begin{table}[h!]\label{tab:high_to_simple}
\caption{Negative highest roots in terms of simple roots}
\begin{tabular}{|c|c|}
\hline
$B_n$ & $-\widetilde{\al_s}=-\al_1-\al_2-\cdots-\al_n$\\
\hline
$C_n$ & $-\widetilde{\al}=-2(\al_1+\cdots+\al_{n-1})-\al_n$\\
\hline
$E_6$ & $-\widetilde{\al}=-\al_1-2\al_2-3\al_3-2\al_4-2\al_5-\al_6$\\
\hline
$E_7$ & $-\widetilde{\al}=-\al_1-2\al_2-3\al_3-4\al_4-2\al_5-3\al_6-2\al_7$\\
\hline
$E_8$ &
$-\widetilde{\al}=-2\al_1-3\al_2-4\al_3-5\al_4-6\al_5-3\al_6-4\al_7-2\al_8$\\
\hline
$F_4$ & $-\widetilde{\al}=-2\al_1-3\al_2-4\al_3-2\al_4$\\
\hline
\end{tabular}
\end{table}

The spin labeled extended Dynkin diagrams that we will need later can now be
found using Lemma~\ref{h_into_simples}, and are exhibited below. Our general
reference for the extended Dynkin diagrams is the Appendix in \cite{carter05}.
Note that the diagram we need for $B_n$ actually has the negative highest
\emph{short} root added.


$$B_n~~~(n>1)
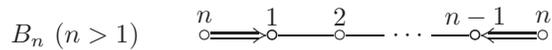
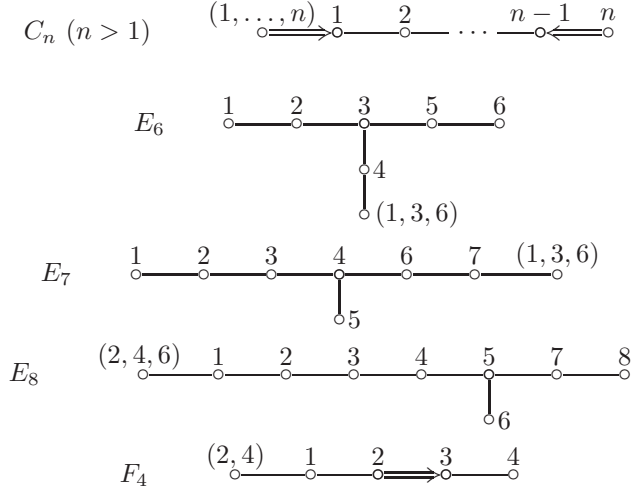
\hspace{.3in}\xy
(-7,2.5)*{n}; (2,2.5)*{1}; (11,2.5)*{2}; (29,2.5)*{n-1}; (38,2.5)*{n};
(2,0)*{\circ}; (11,0)*{\circ} **\dir{-}; (20,0)*{~\cdots~} **\dir{-};
(29,0)*{\circ} **\dir{-}; {\ar@2{<-}(29,0)*{\circ};(38,0)*{\circ}};
{\ar@2{->}(-7,0)*{\circ};(2,0)*{\circ}};
\endxy$$

$$C_n~~~(n>1)\hspace{.3in}\xy
(-8,2.5)*{(1,\dots,n)}; (2,2.5)*{1}; (11,2.5)*{2}; (29,2.5)*{n-1}; (38,2.5)*{n};
(2,0)*{\circ}; (11,0)*{\circ} **\dir{-}; (20,0)*{~\cdots~} **\dir{-};
(29,0)*{\circ} **\dir{-}; {\ar@2{<-}(29,0)*{\circ};(38,0)*{\circ}};
{\ar@2{->}(-8,0)*{\circ};(2,0)*{\circ}};
\endxy$$


$$E_6\hspace{.3in}\xy
(2,2.5)*{1}; (11,2.5)*{2}; (20,2.5)*{3}; (29,2.5)*{5}; (38,2.5)*{6};
(22,-6)*{4}; (27,-12)*{(1,3,6)};
(2,0)*{\circ}; (11,0)*{\circ} **\dir{-}; (20,0)*{\circ} **\dir{-};
(29,0)*{\circ} **\dir{-}; (38,0)*{\circ} **\dir{-}; (20,0)*{\circ};
(20,-6)*{\circ} **\dir{-}; (20,-12)*{\circ} **\dir{-};
\endxy$$

$$E_7\hspace{.3in}\xy
(2,2.5)*{1}; (11,2.5)*{2}; (20,2.5)*{3}; (29,2.5)*{4}; (38,2.5)*{6};
(47,2.5)*{7}; (58,2.5)*{(1,3,6)}; (31,-6)*{5};
(2,0)*{\circ}; (11,0)*{\circ} **\dir{-}; (20,0)*{\circ} **\dir{-};
(29,0)*{\circ} **\dir{-}; (38,0)*{\circ} **\dir{-}; (47,0)*{\circ} **\dir{-};
(58,0)*{\circ} **\dir{-}; (29,0)*{\circ}; (29,-6)*{\circ} **\dir{-};
\endxy$$

$$E_8\hspace{.3in}\xy
(-8.5,2.5)*{(2,4,6)}; (2,2.5)*{1}; (11,2.5)*{2}; (20,2.5)*{3}; (29,2.5)*{4};
(38,2.5)*{5}; (47,2.5)*{7}; (56,2.5)*{8};  (40,-6)*{6};
(-8,0)*{\circ}; (2,0)*{\circ} **\dir{-}; (11,0)*{\circ} **\dir{-};
(20,0)*{\circ} **\dir{-}; (29,0)*{\circ} **\dir{-}; (38,0)*{\circ} **\dir{-};
(47,0)*{\circ} **\dir{-}; (56,0)*{\circ} **\dir{-}; (38,0)*{\circ};
(38,-6)*{\circ} **\dir{-};
\endxy$$

$$F_4\hspace{.3in}\xy
(2,2.5)*{1}; (11,2.5)*{2}; (20,2.5)*{3}; (29,2.5)*{4}; (-8,2.5)*{(2,4)};
(-8,0)*{\circ}; (2,0)*{\circ} **\dir{-}; (11,0)*{\circ} **\dir{-};
{\ar@2{->}(11,0)*{\circ};(20,0)*{\circ}}; (20,0)*{\circ} ; (29,0)*{\circ}
**\dir{-};
\endxy$$

It is certainly possible that a given Carter diagram $\G$ could have more than
one spin labeling. For example the Carter diagram $C_2\times A_1$ in type $C_3$
could have either of the spin labelings below.

$$\xy
(2,2.5)*{(1,2,3)}; (11,2.5)*{1}; (20,2.5)*{3};
{\ar@2{->}(2,0)*{\circ}; (11,0)*{\circ}}; (20,0)*{\circ};
(60,2.5)*{(1,2,3)}; (69,2.5)*{2}; (78,2.5)*{3};
(60,0)*{\circ}; {\ar@2{<-}(69,0)*{\circ}; (78,0)*{\circ}};
\endxy$$
The first is obtained by removing the node labeled ``2" from the extended Dynkin
diagram, and the second by removing the node labeled ``1."

Luckily, as we will see, the spin signature is almost always central and so
different spin labelings will still produce the same spin signature. The example
given here is one of the few for which different spin labelings produce
different spin signatures, namely $h_1$ and $h_2$, as seen in
Section~\ref{sec:C}. In any case, to at least calculate the spin it doesn't
matter which spin labeling we pick for our Carter diagram.

We can now calculate the spin and spin signature of all elliptic elements in any
Weyl group. Let $w\in W$ be elliptic with order $d$, and let $\G$ be its Carter
diagram. Let $\G=\G_1\times\cdots\times\G_r$ be a decomposition of $\G$ into
connected components. We know that $w=w_1\cdots w_r$ where each $w_i$ has Carter
diagram $\G_i$ and all the $w_i$ commute with each other. Let $d_i$ denote the
order of $w_i$, so $d$ is the least common multiple of the $d_i$.

\subsection{The $A_{n-1}$ case}\label{sec:A}
All elliptic elements are Coxeter elements, and so we already calculated their
spin and spin signature in Section~\ref{sec:ACox}.

\subsection{The $B_n$ case}\label{sec:B}
In \cite{fedotov09} it is shown that all elliptic elements in $W=W(B_n)$ have
adjoint spin 1. We now have the tools to calculate the universal spin of any
elliptic element, with the adjoint case as a corollary. Note that these are the
only two cases since $|Z(G_u)|=2$. By \cite{carter72}*{Proposition~24}, each
$\G_i$ is a Dynkin diagram of type $B_{n_i}$ for some $n_i$, and
$n_1+\cdots+n_r=n$. Here $B_1$ is identified with $\widetilde{A}_1$, a single
node corresponding to a short root. By Proposition~24 and Table~2 in
\cite{carter72} each $\G$ arises by an iterated process of attaching a node for
the negative highest short root and removing a node. As seen in the previous
section, these new nodes will all have spin labeling ``$n$." If $g_i$ is the
usual representative of $w_i$ in $N_0$, it is thus easy to calculate $g_i^d$
using Section~\ref{sec:BCox}. The problem though is that the result of
Lemma~\ref{orth_m_commute} does \emph{not} hold, since the negative highest
short roots introduced do not have trivial root chains through each other.
Luckily by Table~\ref{tab:center} $h_n$ is central, and so it is not too
difficult to calculate the spin signature of $w$ explicitly.

Let $g_0\in N_0$ represent $w$. Let $f$ be the number of relevant $\G_i$ such
that $n_i\equiv_41,2$. If $d\equiv_40$ or $r\equiv_40,1$ then set $e:=f$. If
$d\equiv_42$ and $r\equiv_42,3$ then set $e:=f+1$. Note that $d$ is even so
these are the only possibilities.

\begin{theorem}\label{type_B_complete}
With the above setup, $g_0^d=h_n^e$. In particular all elliptic $w$ have adjoint
spin 1.
\end{theorem}
\begin{proof}
For each $i$ let $g_i\in N_0$ be the standard representative given by the
product of $m_{\al}(1)$ as $\al$ ranges over $\G_i$. Without loss of generality
$g_0=g_1\cdots g_r$. By Section~\ref{sec:BCox} and the fact that all nodes of
$\G$ corresponding to short roots have spin labeling ``$n$," it is immediate
that $g_1^d\cdots g_r^d=h_n^f$. We now claim that for any $i\neq j$,
$g_ig_jg_i\I=g_jh_n$. Indeed, if $\G_i$ is labeled by the roots
$\be_1,\dots,\be_{n_i}$ and $\G_j$ by $\ga_1,\dots,\ga_{n_j}$ (with $\be_{n_i}$
and $\ga_{n_j}$ short), then by Lemma~\ref{orth_m_commute} $m_{\be_k}(1)$
commutes with $m_{\ga_{\ell}}(1)$ for all $(k,\ell)\neq(n_i,n_j)$. Also, the
$\be_{n_i}$-chain of roots through $\ga_{n_j}$ is
$\ga_{n_j}-\be_{n_i},\ga_{n_j},\ga_{n_j}+\be_{n_i}$, so
$m_{\be_{n_i}}(1)m_{\ga_{n_j}}(1)m_{\be_{n_i}}(1)\I=m_{\ga_{n_j}}(-1)$. Since
$m_{\ga_{n_j}}(-1)=m_{\ga_{n_j}}(1)h_{\ga_{n_j}}(-1)$ and
$h_{\ga_{n_j}}(-1)=h_n$, in fact
$m_{\be_{n_i}}(1)m_{\ga_{n_j}}(1)m_{\be_{n_i}}(1)\I=m_{\ga_{n_j}}(1)h_n$. This
proves our claim that $g_ig_jg_i\I=g_jh_n$.

It is now a straightforward exercise to calculate $(g_1\cdots g_r)^d$ in terms
of $g_1^d,\dots,g_r^d$. Since $h_n$ is central and $T_0$ is abelian, we get the
following:
\begin{align*}
g_0^d&=(g_rh_n^{r-1}g_rh_n^{2(r-1)}\cdots
g_rh_n^{d(r-1)})\cdots(g_2h_ng_2h_n^2\cdots g_2h_n^d)g_1^d\\
&=g_1^d\cdots g_r^d h_n^{d/2}h_n^{2d/2}\cdots h_n^{(r-1)d/2}\\
&=h_n^f h_n^{\frac{d}{2}(1+2+\cdots+(r-1))}\\
&=h_n^f h_n^{\frac{dr(r-1)}{4}}
\end{align*}
If $d\equiv_40$ or $r\equiv_40,1$, we see that $g_0^d=h_n^f=h_n^e$. If
$d\equiv_42$ and $r\equiv_42,3$ then $g_0^d=h_n^fh_n=h_n^e$.
\end{proof}

\begin{corollary}\label{B_complete_char_poly}
Let $w\in W(B_n)$ be elliptic with characteristic polynomial
$(t^{n_1}+1)\cdots(t^{n_r}+1)$. Then $w$ has spin signature $h_n^e$ where $e$ is
as in Theorem~\ref{type_B_complete}.
\end{corollary}
\begin{proof}
By \cite{carter72}*{Proposition~24} $w$ has Carter diagram of type
$B_{n_1}\times\cdots\times B_{n_r}$, and the result is immediate from
Theorem~\ref{type_B_complete}.
\end{proof}

In summary, all elliptic elements in $W(B_n)$ have adjoint spin 1, and we can
calculate the universal spin just knowing the characteristic polynomial of $w$.
Many conjugacy classes have universal spin 1, and many have universal spin $-1$.
We illustrate this with a few examples.

\begin{example}
Let $w\in W(B_7)$ have characteristic polynomial $(t^3+1)(t^3+1)(t^+1)$. The
spin labeled Carter diagram we use is
$$\xy
(2,2.5)*{7}; (11,2.5)*{1}; (20,2.5)*{2}; (29,2.5)*{7}; (38,2.5)*{4};
(47,2.5)*{5}; (56,2.5)*{7};
{\ar@2{<-}(2,0)*{\circ}; (11,0)*{\circ}}; (11,0)*{\circ}; (20,0)*{\circ}
**\dir{-}; {\ar@2{<-}(29,0)*{\circ}; (38,0)*{\circ}}; (38,0)*{\circ};
(47,0)*{\circ} **\dir{-}; (56,0)*{\circ};
\endxy$$
All three components are relevant since they all have content 2. Since
$n_1=n_2=3$ and $n_3=1$, we have $f=1$. Also, since $d=6$ and $r=3$, we have
$e=f+1=2$. Thus $w$ has spin signature $h_7^2=1$, and so has universal spin 1.
In the language of algebraic groups, any representative of $w$ in $\SO_{15}$ has
order 6, and even in $\Spin_{15}$, any representative has order 6.
\end{example}

\begin{example}
Let $w\in W(B_7)$ have characteristic polynomial $(t^6+1)(t+1)$.  The spin
labeled Carter diagram we use is
$$\xy
(2,2.5)*{7}; (11,2.5)*{1}; (20,2.5)*{2}; (29,2.5)*{3}; (38,2.5)*{4};
(47,2.5)*{5}; (56,2.5)*{7};
{\ar@2{<-}(2,0)*{\circ}; (11,0)*{\circ}}; (11,0)*{\circ}; (20,0)*{\circ}
**\dir{-}; (29,0)*{\circ} **\dir{-}; (38,0)*{\circ} **\dir{-}; (47,0)*{\circ}
**\dir{-}; (56,0)*{\circ};
\endxy$$
Only the first component is relevant, and $n_1=6$, so $f=1$. Also, since $d=12$
and $r=2$ we have $e=f=1$. Thus $w$ has spin signature $h_7$, and so has
universal spin $-1$. In particular any representative of $w$ in $\SO_{15}$ has
order 12 but any representative in $\Spin_{15}$ actually has order 24.
\end{example}

\begin{remark}\label{spin_name}
The name ``spin" is slightly justified now. Indeed, it in some sense measures
the tendency of representatives in $\SO_m$ of elliptic $w$ to pick up an extra
``twist" when lifting to $\Spin_m$, that is, the order doubles. As we have seen
not every $w$ has this property, but we can tell which ones do just based on
their characteristic polynomials, so this really is an inherent property of $w$.
\end{remark}

\subsection{The $C_n$ case}\label{sec:C}
The universal case is covered in \cite{fedotov09}, with the conclusion that all
elliptic elements have universal spin $-1$. While we could realize $G_u$ and
$G_a$ explicitly as $\Sp_{2n}$ and $\PSp_{2n}$, we find that to cover the
general case it is convenient to just deal directly with Carter diagrams. The
result we find is the following

\begin{theorem}\label{type_C_complete}
Let $w\in W(C_n)$ be elliptic. Then $w$ has universal spin $-1$, and has adjoint
spin 1 if and only if $w^r=-I\in W$ for some $r$.
\end{theorem}

As in the $B_n$ case, the Carter diagrams for elements of $W=W(C_n)$ all arise
by removing nodes from extended Dynkin diagrams. This time though, we will use
the negative highest roots instead of the negative highest short roots. Now each
$\G_i$ is $C_{n_i}$ for some $n_i$, and $n_1+\cdots+n_r=n$
\cite{carter72}*{Proposition~24}. (We identify $C_1$ with $A_1$.) Since we only
ever introduce long roots, every short root corresponding to a node of $\G$ must
actually be one of the simple roots $\al_1,\dots,\al_{n-1}$. We claim that if
two roots $\al,\be$ corresponding to nodes of $\G$ are orthogonal, then the
$\al$-chain through $\be$ is just $\be$. This is clear if either $\al$ or $\be$
is long. Also, if both roots are short, they are both simple, and orthogonal
simple roots satisfy this property. In any case, if $\al$ and $\be$ are
orthogonal then by Lemma~\ref{orth_m_commute}, $[m_{\al}(1),m_{\be}(1)]=1$. Let
$g_i$ be the usual representative of $w_i$ and let $g_0=g_1\cdots g_r$ represent
$w$, so $g_0^d=g_1^d\cdots g_r^d$.

\begin{proof}[Proof of Theorem~\ref{type_C_complete}]
First note that the characteristic polynomial of $w_i$ is $t^{n_i}+1$, and so
$w$ is linked to $-I$ if and only if every $\G_i$ is relevant. If $w$ is linked
to $-I$ then by Section~\ref{sec:CCox}, $w$ has spin signature $h_1h_3\cdots
h_k$, so $w$ has adjoint spin 1 and universal spin $-1$. Now suppose $w$ is not
linked to $-I$. We know that $g_0^d$ is the product of the $g_i^d$ ranging over
all $i$ such that $\G_i$ is relevant. Also, for each relevant $\G_i$, $n_i$ must
be even since otherwise all $w_j$ would have order congruent to 2 mod 4,
implying that all $\G_i$ are relevant and $w$ in fact is linked to $-I$. By
Section~\ref{sec:CCox}, $g_i^d$ is a product of $h_{\al}(-1)$ where $\al$ ranges
over every other root of $\G_i$, beginning with the terminal short root. Also,
since $n_i$ is even for relevant $\G_i$, all such $\al$ are short roots and thus
simple roots. This tells us that $g_0^d$ is a product of $h_{\al}(-1)$ as $\al$
ranges over every simple root contained in a relevant $\G_i$. Such an $i$
exists, and so immediately we see that $g_0^d\neq1$, and $w$ has universal spin
-1. By Table~\ref{tab:center}, it now suffices to show that for some
$j=1,3,\dots,k$, the simple root $\al_j$ is not a node in any relevant $\G_i$.

Indeed, since $w$ is not linked to $-I$ we know there exists \emph{some}
irrelevant $\G_i$. The only way $\G_i$ can avoid containing a node $\al_j$ for
odd $j$ is if $n_i=2$ and the two nodes of $\G_i$ are a long root and some
$\al_{\ell}$ for even $\ell$. But then one of $\al_{\ell+1}$ or $\al_{\ell-1}$
must have been removed from the graph, or else $\G_i$ would not be a connected
component of $\G$. We conclude that $g_0^d$ cannot equal $h_1h_3\cdots h_k$, and
so $w$ has adjoint spin $-1$.
\end{proof}

\begin{remark}\label{signatures_in_C}
The last paragraph of the proof does not explicitly calculate the spin signature
$g_0^d$, and indeed since the spin signatures are non-central, conjugate Weyl
group elements may have different spin signatures.
\end{remark}

\begin{example}
In $C_6$, consider the conjugacy class with Carter diagram $C_2\times C_4$. Any
corresponding element $w$ has order 8. One spin labeling of the Carter diagram
is
$$\xy
(-4,2.5)*{(1,\dots,6)}; (9,2.5)*{1}; (20,2.5)*{3}; (29,2.5)*{4}; (38,2.5)*{5};
(47,2.5)*{6};
{\ar@2{->}(-4,0)*{\circ}; (9,0)*{\circ}}; (20,0)*{\circ}; (29,0)*{\circ}
**\dir{-}; (38,0)*{\circ} **\dir{-}; {\ar@2{<-}(38,0)*{\circ}; (47,0)*{\circ}};
\endxy$$
and only the $C_4$ component is relevant, so $g_0^8=h_3h_5$, which is not
central. Thus $w$ has spin $-1$, even in the adjoint case.
\end{example}

\begin{example}
In $C_8$, consider an element $w$ with Carter diagram $C_2\times C_6$ with spin
labeling
$$\xy
(-4,2.5)*{(1,\dots,8)}; (9,2.5)*{1}; (20,2.5)*{3}; (29,2.5)*{4}; (38,2.5)*{5};
(47,2.5)*{6}; (56,2.5)*{7}; (65,2.5)*{8};
{\ar@2{->}(-4,0)*{\circ}; (9,0)*{\circ}}; (20,0)*{\circ}; (29,0)*{\circ}
**\dir{-}; (38,0)*{\circ} **\dir{-}; (47,0)*{\circ} **\dir{-}; (56,0)*{\circ}
**\dir{-}; {\ar@2{<-}(56,0)*{\circ}; (65,0)*{\circ}};
\endxy$$
and order 12.
Now both components are relevant, so $g_0^{12}=h_1h_3h_5h_7$, which is a
nontrivial element of $Z(G_u)$. Thus $w$ has adjoint spin 1 and universal spin
-1.
\end{example}

It turns out the $C_n$ case provides the only source of elliptic Weyl group
elements with adjoint spin $-1$, except for one conjugacy class in $F_4$. The
$C_n$ case is also the only case where \emph{every} elliptic element has
universal spin $-1$.

\subsection{The $D_n$ case}\label{sec:D}
Certain cases are essentially done in \cite{fedotov09}, though there are no
results there for the universal case. Unfortunately, type $D_n$ is the only
classical type in which not every elliptic Carter diagram arises by removing
nodes from extended Dynkin diagrams, and applying our present approach to
diagrams with cycles would be very difficult. However, having completely handled
the $B_n$ case we can now just use the natural embeddings $W(D_n)\leq W(B_n)$
and $G_u(D_n)\leq G_u(B_n)$ to figure out the spin of any $w\in W(D_n)$. Indeed,
if $w\in W(D_n)$ is elliptic then it is also an elliptic element in $W(B_n)$, so
we know its spin and spin signature in $G_u(B_n)$ just from its characteristic
polynomial. Then since all the spin signatures are central they are independent
of the choice of representative $g_0$, and we can choose a representative in
$G_u(D_n)$, which tells us the spin and spin signature in $G_u(D_n)$, though we
have to use Lemma~\ref{h_into_simples} to express the spin signature in the
correct notation. As an example we show the case of Coxeter elements in
$W(D_n)$.

\begin{example}\label{DCox_example}
Let $w\in W=W(D_n)$ be a Coxeter element. Then $w$ has characteristic polynomial
$(t^{n-1}+1)(t+1)$, and so as an element of $W'=W(B_n)$, $w$ has spin labeled
Carter diagram
$$\xy
(2,2.5)*{n}; (11,2.5)*{1}; (20,2.5)*{2}; (38,2.5)*{n-1}; (47,2.5)*{n};
{\ar@2{<-}(2,0)*{\circ}; (11,0)*{\circ}}; (11,0)*{\circ}; (20,0)*{\circ}
**\dir{-}; (29,0)*{~\cdots~} **\dir{-}; (38,0)*{\circ} **\dir{-};
(47,0)*{\circ};
\endxy$$
So as not to confuse central elements of $G_u(B_n)$ and $G_u(D_n)$ we will use
$\widetilde{h}_n$ for the central element of $G_u(B_n)$. Direct calculation
shows that $e$ equals 0, 1, 2, or 3, if $n$ is congruent modulo 4 to 1, 3, 0, or
2, respectively. Thus the spin signature of $w$ in $G_u(B_n)$ is either 1 or
$\widetilde{h}_n$, if $n\equiv_40,1$ or $n\equiv_42,3$, respectively.

Now to figure out the spin signature of $w$ in $G_u(D_n)$ we need to calculate
$\widetilde{h}_n$ in terms of the $h_i$. We know that $\Phi(D_n)$ is the subroot
system of $\Phi(B_n)$ consisting of the long roots, with fundamental roots
$\al_1,\al_2,\dots,\al_{n-1},-\widetilde{\al}$. The root $\al_n$ equals
$(-\widetilde{\al}-\al_1-2(\al_2+\cdots+\al_{n-1}))/2$, and so by
Lemma~\ref{h_into_simples},
$\widetilde{h}_n=\widetilde{h}_{-\widetilde{\al}}(-1)\widetilde{h}_1$.
Converting to the standard numbering of fundamental roots in $\Phi(D_n)$, this
equals $h_{n-1}h_n$, one of the central elements in $G_u(D_n)$.
\end{example}

\begin{corollary}\label{DCox}
Coxeter elements in $W(D_n)$ have adjoint spin 1, and have universal spin 1 or
-1, if $n\equiv_40,1$ or $n\equiv_42,3$, respectively. Moreover, if the
universal spin is -1 then the spin signature is $h_{n-1}h_n$.\qed
\end{corollary}

\begin{remark}
Note that if $n$ is even, $G_u$ has two central elements of order 2 other than
$h_{n-1}h_n$, but they will never appear as spin signatures of elliptic
elements.
\end{remark}

\subsection{The $G_2$ case}\label{sec:G2}
The $G_2$ case was completely dealt with in \cite{fedotov09} using a different
method, but we will present it for completeness. The Weyl group $W=W(G_2)$ is
just the dihedral group of order 12. The Coxeter element $w$ is the rotation of
order 6, and a complete list of elliptic elements is given by
$w,w^2,w^3,w^4,w^5$; in particular they are all linked. Since $w^2$ has order 3
it has spin 1 by Corollary~\ref{odd_order}, and so by Lemma~\ref{powers_spin}
all elliptic elements have spin 1.

\subsection{The $F_4$ case}\label{sec:F4}
The $F_4$ case is partially covered in \cite{fedotov09}, in particular it is
shown that any elliptic power of a Coxeter element has spin 1. Here we show that
one elliptic conjugacy class has spin $-1$ and all others have spin 1. First
note that $-I\in W=W(F_4)$ and $G$ is simple, so by Corollary~\ref{simple_-I}
any elliptic $w$ linked to $-I$ will have spin 1. By Carter's classification in
\cite{carter72}, there are 9 elliptic conjugacy classes in $W$, and inspecting
Tables~3 and 8 in \cite{carter72} it is clear that 7 of these are linked to
$-I$. The two remaining classes have Carter diagram $A_2\times\widetilde{A}_2$
and $A_3\times\widetilde{A}_1$, where a tilde indicates the roots labeling the
nodes are short. Elements corresponding to the first diagram have odd order and
thus spin 1 by Corollary~\ref{odd_order}. This leaves the single class with
diagram $A_3\times\widetilde{A}_1$ having unknown spin. Let $w$ be an element of
this class, so $w$ has order 4 and spin labeled diagram
$$\xy
(2,2.5)*{(2,4)}; (11,2.5)*{1}; (20,2.5)*{2}; (29,2.5)*{4};
(2,0)*{\circ}; (11,0)*{\circ} **\dir{-}; (20,0)*{\circ} **\dir{-};
(29,0)*{\circ};
\endxy$$
Inspecting the root system for $F_4$, it is clear that the $\al_4$-chain of
roots through any of $\al_1,\al_2,-\widetilde{\al}$ consists of a single root.
The same is true of $\al_2$ and $-\widetilde{\al}$, and so the conclusion of
Lemma~\ref{orth_m_commute} holds. Since the $\widetilde{A}_1$ component is
irrelevant, this implies that $g_0^4=(h_2h_4)h_2=h_4$. We conclude that $w$
actually has spin $-1$ in this case. Also note that different spin labelings may
yield different (through still non-trivial) spin signatures.

\subsection{The $E_6$ case}\label{sec:E6}
This is the first case for which no results were found in \cite{fedotov09}.
However, we can now completely handle this case, with the result that all
elliptic elements have spin 1. By Carter's classification in \cite{carter72},
there are 5 elliptic conjugacy classes in $W=W(E_6)$, and inspecting Tables~3
and 9 in \cite{carter72} it is clear that 4 of these are linked to the class
with Carter diagram $A_2^3$. Elements of this class have order 3, and so have
spin 1 by Corollary~\ref{odd_order}. This leaves only the class with diagram
$A_1\times A_5$ having unknown spin. Let $w$ be a representative of this class,
so $w$ has order 6. A spin labeling of the Carter diagram is
$$\xy
(0,2.5)*{1}; (10,2.5)*{(1,3,6)}; (20,2.5)*{4}; (29,2.5)*{3}; (38,2.5)*{5};
(47,2.5)*{6};
(0,0)*{\circ}; (10,0)*{\circ}; (20,0)*{\circ} **\dir{-}; (29,0)*{\circ}
**\dir{-}; (38,0)*{\circ} **\dir{-}; (47,0)*{\circ} **\dir{-};
\endxy$$
Since $E_6$ is simply laced and both components are relevant,
$$g_0^6=h_1h_1h_3h_6h_3h_6=1.$$
Thus all elliptic elements in $W(E_6)$ have spin 1, in both the adjoint and
universal case.

\subsection{The $E_7$ case}\label{sec:E7}
Like $E_6$, no results were found in \cite{fedotov09} for the $E_7$ case. Since
$|Z(G_u)|=2$, $G$ must be either $G_a$ or $G_u$. In Table~10 of \cite{carter72}
the elliptic conjugacy classes in $W=W(E_7)$ are classified, and in Table~3 in
\cite{carter72} the corresponding characteristic polynomials are given, so we
can tell which elliptic elements are linked to each other. If $w_1,\dots,w_{12}$
denote choices of representatives of each elliptic conjugacy class, in the order
given in \cite{carter72}, then $w_1=-I$ is linked to $w_i$ for
$i=5,6,7,8,9,11,12$, and $w_2$ is linked to $w_{10}$. By Section~\ref{sec:E7Cox}
then, $w_i$ has universal spin $-1$ and adjoint spin 1 for
$i=1,5,6,7,8,9,11,12$. We now determine the spin of $w_i$ for $i=2,3,4,10$.

First consider $w_2$, which has order 4, Carter diagram $A_3^2\times A_1$, and
spin labeling
$$\xy
(2,2.5)*{1}; (11,2.5)*{2}; (20,2.5)*{3}; (29,2.5)*{5}; (38,2.5)*{6};
(47,2.5)*{7}; (57,2.5)*{(1,3,6)};
(2,0)*{\circ}; (11,0)*{\circ} **\dir{-}; (20,0)*{\circ} **\dir{-};
(29,0)*{\circ}; (38,0)*{\circ}; (47,0)*{\circ} **\dir{-}; (57,0)*{\circ}
**\dir{-};
\endxy$$
The $A_1$ component is irrelevant, so $g_0^4$ equals $h_1h_3h_6(h_1h_3h_6)$,
which is 1. Thus $w_2$ (and consequently $w_{10}$) has spin 1.

Next consider $w_3$, with order 6, Carter diagram $A_5\times A_2$, and spin
labeling
$$\xy
(2,2.5)*{1}; (11,2.5)*{2}; (20,2.5)*{3}; (29,2.5)*{4}; (38,2.5)*{5};
(47,2.5)*{7}; (57,2.5)*{(1,3,6)};
(2,0)*{\circ}; (11,0)*{\circ} **\dir{-}; (20,0)*{\circ} **\dir{-};
(29,0)*{\circ} **\dir{-}; (38,0)*{\circ} **\dir{-}; (47,0)*{\circ};
(57,0)*{\circ} **\dir{-};
\endxy$$
The $A_2$ component is irrelevant, so $g_0^6=h_1h_3h_5$. But this is precisely
the non-trivial element of $Z(G_u)$, by Table~\ref{tab:center}. So $w_3$ has
universal spin $-1$ and adjoint spin 1.

Lastly consider $w_4$, with order 8, Carter diagram $A_7$, and spin labeling
$$\xy
(2,2.5)*{1}; (11,2.5)*{2}; (20,2.5)*{3}; (29,2.5)*{4}; (38,2.5)*{6};
(47,2.5)*{7}; (57,2.5)*{(1,3,6)};
(2,0)*{\circ}; (11,0)*{\circ} **\dir{-}; (20,0)*{\circ} **\dir{-};
(29,0)*{\circ} **\dir{-}; (38,0)*{\circ} **\dir{-}; (47,0)*{\circ} **\dir{-};
(57,0)*{\circ} **\dir{-};
\endxy$$
Then $g_0^8=h_1h_3h_6(h_1h_3h_6)=1$, so $w_4$ has spin 1.

In conclusion, $w_2$, $w_4$ and $w_{10}$ always have spin 1, and all other
elliptic elements have adjoint spin 1 and universal spin $-1$.

\subsection{The $E_8$ case}\label{sec:E8}
As with the $F_4$ case, in \cite{fedotov09} it is indicated that powers of
Coxeter elements have spin 1. Here we show that all elliptic elements have spin
1. Note that $-I\in W=W(E_8)$ and $G$ is simple, so by Corollary~\ref{simple_-I}
any elliptic $w$ linked to $-I$ will have spin 1. By Carter's classification in
\cite{carter72}, there are 30 elliptic conjugacy classes in $W$. Choose
representatives for each class, denoted $w_1,\dots,w_{30}$. Inspecting Table~3
in \cite{carter72} it is clear that $w_i$ is linked to $w_1$ for all $i$ except
for $i=2,3,4,5,6,7,11,17,18,20,29$. Since $w_1=-I$ the classes linked to $w_1$
all have spin 1. Moreover for $i=2,4,7$, $w_i$ has odd order so $w_i$ has spin
1, and for $i=17,18,29$, $w_i$ is linked to $w_2$ and so has spin 1. Also, the
$w_i$ for $i=3,11,20,29$ are all linked to $w_3$, and $w_{29}$ has spin 1 so all
these $w_i$ do too. This leaves $w_5$ and $w_6$ as the only remaining cases,
which we can handle using spin labelings.

First consider $w_5$, with Carter diagram $A_5\times A_1\times A_2$ and spin
labeling
$$\xy
(0,2.5)*{(2,4,6)}; (11,2.5)*{1}; (20,2.5)*{2}; (29,2.5)*{3}; (38,2.5)*{4};
(47,2.5)*{6}; (56,2.5)*{7}; (65,2.5)*{8};
(0,0)*{\circ}; (11,0)*{\circ} **\dir{-}; (20,0)*{\circ} **\dir{-};
(29,0)*{\circ} **\dir{-}; (38,0)*{\circ} **\dir{-}; (47,0)*{\circ};
(56,0)*{\circ}; (65,0)*{\circ} **\dir{-};
\endxy$$
The $A_2$ component is irrelevant, so $g_0^6=(h_2h_4h_6)h_2h_4h_6=1$ and $w_5$ has
spin 1.

Lastly consider $w_6$, with Carter diagram $A_7\times A_1$ and order 8 and spin
labeling
$$\xy
(0,2.5)*{(2,4,6)}; (11,2.5)*{1}; (20,2.5)*{2}; (29,2.5)*{3}; (38,2.5)*{4};
(47,2.5)*{5}; (56,2.5)*{6}; (65,2.5)*{8};
(0,0)*{\circ}; (11,0)*{\circ} **\dir{-}; (20,0)*{\circ} **\dir{-};
(29,0)*{\circ} **\dir{-}; (38,0)*{\circ} **\dir{-}; (47,0)*{\circ} **\dir{-};
(56,0)*{\circ} **\dir{-}; (65,0)*{\circ};
\endxy$$
The $A_1$ component is irrelevant, so $g_0^8=(h_2h_4h_6)h_2h_4h_6=1$ and $w_6$ has
spin 1.

It is remarkable that outside some cases in $C_n$ and one case in $F_4$, every
elliptic conjugacy class has adjoint spin 1. We also see that universal spin
$-1$ occurs all the time in $C_n$, half the time in $A_{n-1}$, never in $E_6$,
most of the time in $E_7$, and quite often in $B_n$ and $D_n$. It seems possible
that these results could be proved without appealing the classification at all,
but at present there is no general method that can handle every case. The
following table summarizes our results:

\begin{table}\label{tab:final_chart}
\caption{Spins of elliptic elements}
\begin{tabular}{|c|c|c|c|c|}
	\hline
	$\Phi$ & $\G$ & \textnormal{adjoint spin} & \textnormal{universal spin} \\
	\hline
	$A_{n-1}$ & $A_{n-1}$ & 1 & $(-1)^{n-1}$ \\
	\hline
	$B_n$ & \textnormal{any} & 1 & see Section~\ref{sec:B} \\
	\hline
	$C_n$ & \textnormal{linked to }$A_1^n$ & 1 & -1\\
	$C_n$ & \textnormal{all others} & -1 & -1 \\
	\hline
	$D_n$ & \textnormal{any} & 1 & see Section~\ref{sec:D} \\
	\hline
	$G_2$ & \textnormal{any} & 1 & 1 \\
	\hline
	$F_4$ & $A_3\times\widetilde{A}_1$ & -1 & -1 \\
	$F_4$ & \textnormal{all others} & 1 & 1 \\
	\hline
	$E_6$ & \textnormal{any} & 1 & 1 \\
	\hline
	$E_7$ & $A_1\times A_3^2$ & 1 & 1 \\
	$E_7$ & $A_7$ & 1 & 1 \\
	$E_7$ & $E_7(a_2)$ & 1 & 1 \\
	$E_7$ & \textnormal{all others} & 1 & -1 \\
	\hline
	$E_8$ & \textnormal{any} & 1 & 1 \\
	\hline
\end{tabular}
\end{table}

\renewcommand{\baselinestretch}{1.2}

\end{document}